\documentclass[letterpaper]{article}
\usepackage[english]{babel}
\usepackage{spconf, graphicx}
\usepackage{amsmath, amsfonts, amssymb, amsthm}
\usepackage{url}
\usepackage[english]{babel}
\usepackage{bm}
\usepackage{xcolor}
\usepackage{caption,subcaption} 
\usepackage[nameinlink, capitalise, noabbrev]{cleveref}

\DeclareMathOperator*{\argmin}{arg\,min}

\DeclareMathOperator*{\diag}{diag}

\newtheorem{theorem}{Theorem}

\usepackage[ruled, vlined, linesnumbered]{algorithm2e}
\newcommand\widebar[1]{\mathop{\overline{#1}}}

\graphicspath{{./figures/}}

\newcommand{\reffig}[1]{Fig.\,\ref{#1}}
\addto\captionsenglish{\renewcommand{\figurename}{Fig.}}

\newcommand\restr[2]{{
  \left.\kern-\nulldelimiterspace 
  #1 
  \vphantom{\big|} 
  \right|_{#2} 
  }}

\title{Region-based motion-compensated iterative reconstruction technique for dynamic computed tomography}
\name{Anh-Tuan Nguyen, Jens Renders, Jan Sijbers, and Jan De Beenhouwer}
\address{imec-Vision Lab, Department of Physics, and $\mu$NEURO Research Centre of Excellence \\ University of Antwerp, Universiteitsplein 1, 2610 Wilrijk, Belgium}

\begin{document}

\pdfimageresolution=300

\pagestyle{empty} 

\maketitle
\newpage
This work has been submitted to the IEEE for possible publication. Copyright may be transferred without notice, after which this version may no longer be accessible.
\begin{abstract}
Current state-of-the-art motion-based dynamic computed tomography reconstruction techniques estimate the deformation by considering motion models in the entire object volume although occasionally the proper change is local. In this article, we address this issue by introducing the region-based Motion-compensated Iterative Reconstruction Technique (rMIRT). It aims to accurately reconstruct the object being locally deformed during the scan, while identifying the deformed regions consistently with the motion models. Moreover, the motion parameters that correspond to the deformation in those areas are also estimated. In order to achieve these goals, we consider a mathematical optimization problem whose objective function depends on the reconstruction, the deformed regions and the motion parameters. The derivatives towards all of them are formulated analytically, which allows for efficient reconstruction using gradient-based optimizers. To the best of our knowledge, this is the first iterative reconstruction method in dynamic CT that exploits the analytical derivative towards the deformed regions.
\end{abstract}


\section{Introduction} 
Dynamic computed tomography is a major part of CT imaging that studies the structure of dynamic objects in CT scans. State-of-the-art motion-based dynamic CT reconstruction techniques mostly consider the motion in the entire object volume (e.g., \cite{Zehni20, Nguyen22_SPIE}), while in real applications (e.g., lung tissue \cite{Soliman16}) only local regions are deformed. This concern was recently mentioned in \cite{Ruymbeek20}. Although several former reconstruction methods were designed to estimate the deformed regions (e.g., \cite{VanEyndhoven14, VanEyndhoven15, Kazantsev15}), those methods do not consider affine motion models.

In this paper, we consider a dynamic CT model for which, in the object volume, there are local regions deformed by affine motion models, while the complementary regions that remain static during the entire acquisition scan. We then propose an iterative method that aims not only to accurately reconstruct the scanned object that contains these locally deformed regions, but also to identify them. Furthermore, the motion parameters corresponding to the deformation are estimated simultaneously with the reconstruction and region estimation. The contributions are summarized as follows:
\begin{itemize}
    \item Mathematical formulation of the class of dynamic CT problems that consider affine motions, which model the deformation in local areas characterized by corresponding binary masks.
    \item Gradient method that aims to minimize an objective function that depends on the reconstruction, the motion parameters and the deformed regions, whose partial derivatives towards all of them are formulated analytically.
    \item The biconvexity of the objective function towards the reconstruction and the locally deformed regions that supports the convergence of the iterative schemes in the proposed gradient method. 
\end{itemize}

\section{Proposed method}
A dynamic CT image can be represented as a sequence of $n$ images $\bm{x}_1$, $\bm{x}_2$, ..., $\bm{x}_n$, each representing the object at a given point in time. The acquisition can be observed as a collection of finite subscans, where the object is assumed to be static during each subscan. Here, a subscan refers to one or more consecutively acquired projections. This procedure can be mathematically modelled as $n$ systems of linear equations:  
\begin{equation}\label{eq:forward_model_subscan}
    \bm{W}_i \bm{x}_i = \bm{b}_i, \hbox{ for } i = 1,..., n,
\end{equation}
where $\bm{W}_i$ and $\bm{b}_i$ are the projection operator and the projection data corresponding to the $i^{th}$ subscan, respectively. These may be interpreted as a single system of the forward model:
\begin{equation}\label{eq:forward_model}
    \begin{bmatrix}
        \bm{W}_1 &   0      & 0      & 0 \\
          0      & \bm{W}_2 & 0      & 0 \\
          0      &   0      & \ddots & 0\\
          0      &   0      &     0  & \bm{W}_n
    \end{bmatrix}
    \begin{bmatrix}
        \bm{x}_1 \\
        \bm{x}_2 \\
        \vdots \\
        \bm{x}_n
    \end{bmatrix}
    =
    \begin{bmatrix}
        \bm{b}_1 \\
        \bm{b}_2 \\
        \vdots \\
        \bm{b}_n
    \end{bmatrix}.
\end{equation}

Let $\bm{\alpha}_i \in \left \{0, 1 \right \}^N$ be a binary mask, which encodes the local region of the unknown original image $\bm{x} \in \left[0, 1 \right]^N$ that appears deformed in the image $\bm{x}_i$. Assume the deformation can be modelled by an affine motion model $M$ that depends on the motion parameter $\bm{p}_i \in \mathbb{R}^M$, the deformed object in the $i^{th}$ subscan can be modelled as follows:
\begin{equation}\label{eq:deformed_image}
    \bm{x}_{i} = {{\widebar{\bm{\alpha}_i}} \circ \bm{x}} + M\left(\bm{p}_i \right) {\left(\bm{\alpha}_i \circ \bm{x} \right)},
\end{equation}
where $\widebar{\bm{\alpha}_i} := \bm{1} - \bm{\alpha}_i$ and $\circ$ is the commutative Hadamard product. In this model, the static part ${\widebar{\bm{\alpha}_i}} \circ \bm{x}$ of $\bm{x}$ remains conserved in the deformed object $\bm{x}_i$, while the dynamic part ${\bm{\alpha}_i} \circ \bm{x}$ appears distorted under the motion model $M$.\\
\\
By substituting the equation \eqref{eq:deformed_image} to \eqref{eq:forward_model} for all $n$, the forward model of the entire projection data then becomes:
\begin{equation}
    \begin{bmatrix}
        \bm{W}_1 & 0 & 0 & 0 \\
        0 & \bm{W}_2 & 0 & 0 \\
        0 & 0 & \ddots & 0 \\
        0 & 0 & 0 & \bm{W}_n
    \end{bmatrix}
    \begin{bmatrix}
        {\widebar{\bm{\alpha}_1}}  \circ \bm{x} + M\left(\bm{p_1} \right) \left(\bm{\alpha _1} \circ \bm{x} \right)\\
        {\widebar{\bm{\alpha}_2}}  \circ \bm{x} + M\left(\bm{p_2} \right) \left(\bm{\alpha _2} \circ \bm{x} \right)\\
        \vdots \\
        {\widebar{\bm{\alpha}_n}}  \circ \bm{x} + M\left(\bm{p_n} \right)\left(\bm{\alpha _n} \circ  \bm{x} \right)
    \end{bmatrix}
    = 
    \begin{bmatrix}
        \bm{b}_1 \\
        \bm{b}_2 \\
        \vdots \\
        \bm{b}_n
    \end{bmatrix}.
\end{equation}
This can be concisely rewritten as a single system:
\begin{equation}\label{eq:rMIRT_equation}
    \bm{W}\left\{ {\widebar{\bm{\alpha}}} [\circ] \bm{x} + \bm{M}\left(\bm{p}\right) \left(\bm{\alpha} [\circ] \bm{x}\right) \right\} = \bm{b},
\end{equation}
where 
\begin{equation}
    \bm{W} = 
    \begin{bmatrix}
        \bm{W}_1 & 0 & 0 & 0 \\
        0 & \bm{W}_2 & 0 & 0 \\
        0 & 0 & \ddots & 0 \\
        0 & 0 & 0 & \bm{W}_n
    \end{bmatrix},
\end{equation}
\begin{equation}
    \bm{\alpha} = 
    \begin{bmatrix}
        \bm{\alpha}_1 \\
        \bm{\alpha}_2 \\
        \vdots \\
        \bm{\alpha}_n
    \end{bmatrix}, 
    \bm{p} = 
    \begin{bmatrix}
        \bm{p}_1 \\
        \bm{p}_2 \\
        \vdots \\
        \bm{p}_n
    \end{bmatrix},
    \bm{b} = 
    \begin{bmatrix}
        \bm{b}_1 \\
        \bm{b}_2 \\
        \vdots \\
        \bm{b}_n
    \end{bmatrix},
\end{equation}
$[\circ]$ is the modified version of the penetrating face product \cite{Slyusar99} between the two column vectors $\bm{\alpha} \in \left\{0, 1\right \}^{nN}$ and $\bm{x} \in \left[0, 1 \right]^N$ defined by
\begin{equation}
    \bm{\alpha} [\circ] \bm{x} = 
    \begin{bmatrix}
        \left[\bm{\alpha}_1 \circ \bm{x} \right]^T, \left[\bm{\alpha}_2 \circ \bm{x} \right]^T, \hdots, \left[\bm{\alpha}_n \circ \bm{x} \right]^T 
    \end{bmatrix}^T,
\end{equation}
and
\begin{equation}
    \bm{M}(\bm{p}) = 
    \begin{bmatrix}
        M(\bm{p}_1) & 0 & 0 & 0 \\
        0 & M(\bm{p}_2) & 0 & 0 \\
        0 & 0 & \ddots & 0 \\
        0 & 0 & 0 & M(\bm{p}_n)
    \end{bmatrix}.
\end{equation}

In order to solve the equation \eqref{eq:rMIRT_equation}, let us consider the following constrained optimization problem as a modified and extended version of \cite{Zehni20, Nguyen22_SPIE, VanEyndhoven12}:
\begin{equation}\label{rMIRT_problem}
    \left[\bm{x}^*, \bm{\alpha}^*, \bm{p}^*\right] = \argmin_{\bm{x} \in \left [0, 1 \right]^N, \bm{\alpha} \in \left \{0, 1 \right \}^{nN}, \bm{p} \in \mathbb{R}^{nM}} f \left(\bm{x}, \bm{\alpha}, \bm{p} \right),
\end{equation}
where 
\begin{equation}\label{eq:rMIRT_objective_function}
    f \left(\bm{x}, \bm{\alpha}, \bm{p}\right) = \frac{1}{2}\left\|\bm{W}\left\{{\widebar{\bm{\alpha}}}  [\circ] \bm{x} + \bm{M}\left(\bm{p}\right) \left(\bm{\alpha} [\circ] \bm{x}\right) \right\} - \bm{b}\right \|_2^2.
\end{equation}

The problem \eqref{rMIRT_problem} can be solved by the iterative schemes presented in Algorithm 1 with the intermediate approximated value of $\bm{\alpha}$ is projected onto the set $\mathcal{S} \equiv \left \{0, 1 \right \}^{nN}$ to obtain the intermediate deformed regions, after which the center of motion is updated.

\begin{algorithm}[h!]\label{alg:algorithm_rMIRT}
\DontPrintSemicolon
\KwIn{Projection $\bm{b}$, projector $\bm{W}$, motion model $\bm{M}$, $\bm{p}^0 \equiv \text{motion parameters in the static case}$, $\bm{x}^0 \equiv \text{motion-uncompensated reconstruction}$, $\bm{\alpha}^0 \equiv \text{observed dynamic region encoder}$, number of iterations $n_{iter}$.}
\KwOut{Reconstruction with region-based motion compensation, locally deformed regions, motion parameters.}

\SetKwBlock{kwFor}{for}{end for}
\kwFor($i = 0:n_{iter}-1$)
{
    $\bm{x}^{i+1} = \bm{x}^{i} - \gamma_{\bm{x}}^i \nabla_{\bm{x}} f \left(\bm{x}^i, \bm{\alpha}^{i}, \bm{p}^i \right)$\\
    $\bm{p}^{i+1} = \bm{p}^{i} - \gamma_{\bm{p}}^i \nabla_{\bm{p}} f \left(\bm{x}^i, \bm{\alpha}^{i}, \bm{p}^i \right)$\\
    $\bm{\alpha}^{i+1} = \bm{\alpha}^{i} - \gamma_{\bm{\alpha}}^i \nabla_{\bm{\alpha}} f \left(\bm{x}^i, \bm{\alpha}^{i}, \bm{p}^i \right)$\\
    Update the center of motion from $\text{Proj}_{\mathcal{S}}\left(\bm{\alpha}^{i+1}\right)$\\
}
\caption{rMIRT}
\end{algorithm}

\noindent The gradient of the objective function is analytically given by $\nabla f = \left[ \left [ \nabla_{\bm{x}} f \right ]^T, \left [ \nabla_{\bm{\alpha}} f \right ]^T, \left [ \nabla_{\bm{p}} f \right ]^T   \right ]^T$, with
\begin{align}\nonumber
    \nabla_{\bm{x}} f =& \left\{ \left[
    \left( \bm{M}(\bm{p}) - \bm{I} \right) \diag \left\{ \bm{\alpha} \right \} + \bm{I} \right]
    \underbrace{
    \begin{bmatrix}
        I \\
        I \\
        \vdots \\
        I
    \end{bmatrix}}_\text{$n$ blocks $I$} \right \}^T \\
    & \times \bm{W}^T \bm{r}, \\ \nonumber
    \nabla_{\bm{\alpha}} f =& \left \{ \left[\bm{M}(\bm{p}) - \bm{I} \right] \underbrace{
    \begin{bmatrix}
    \diag \left\{\bm{x} \right \} & 0 & 0 \\
    0 & \ddots & 0 \\
    0 & 0 & \diag \left\{\bm{x} \right \}
    \end{bmatrix}}_\text{$n$ blocks $\diag \left\{\bm{x} \right \}$} \right \}^T \\
    & \times \bm{W}^T \bm{r}, \\
    \nabla_{\bm{p}} f =& \left[\nabla \bm{M}(\bm{p})\left(\bm{\alpha} [\circ] \bm{x} \right)\right]^T\bm{W}^T\bm{r}, \mbox{ where}
\end{align}
\begin{equation}
\diag \left \{ \bm{\alpha} \right \} := 
    \begin{bmatrix}
    \diag \left \{ \bm{\alpha}_1 \right \} & 0 & 0 \\
    0 & \ddots & 0 \\
    0 & 0 & \diag \left \{ \bm{\alpha}_n \right \}
    \end{bmatrix},
\end{equation}
\begin{equation}
\bm{I} = \underbrace{
    \begin{bmatrix}
    I & 0 & 0 & 0 \\
    0 & I & 0 & 0 \\
    0 & 0 & \ddots & 0 \\
    0 & 0 & 0 & I
    \end{bmatrix}}_\text{$n$ blocks $I$},
\end{equation}
and $\bm{r}$ is the residual of the system \eqref{eq:rMIRT_equation} given as the following:
\begin{equation}\label{rMIRT_residual}
    \bm{r} = \bm{W}\left\{ {\widebar{\bm{\alpha}}} [\circ] \bm{x} + \bm{M}\left(\bm{p}\right) \left(\bm{\alpha} [\circ] \bm{x}\right) \right\} - \bm{b}.
\end{equation}
The operators $\bm{M} \left( \bm{p} \right)$, $\bm{M} \left(\bm{p}\right)^T$ and $\nabla \bm{M}\left(\bm{p} \right)$ are all provided by a matrix-free and GPU-accelerated implementation of cubic image warping, its adjoint and its derivatives \cite{Renders21} designed to study continuous and differentiable affine motions. The operators $\bm{W}$ and $\bm{W}^T$ of the CT system are provided by the ASTRA Toolbox \cite{VanAarle15}.

The objective function of the proposed method is non-convex towards the motion parameters $\bm{p}$. Nonetheless, the convergence of the iterative parameter estimation scheme was validated in \cite{ Nguyen22_SPIE}. The convergence of the reconstruction and region encoder estimation schemes is supported by the following property, which shows the biconvexity of the objective function towards the reconstruction $\bm{x}$ and the region encoder $\bm{\alpha}$.

\begin{theorem}
Let us assume the domain of the objective function $f$ is extended to $\left[0, 1 \right]^{nN}$, $f$ is then biconvex towards the reconstruction $\bm{x}$ and the region encoder $\bm{\alpha}$. 
\end{theorem}

\begin{proof}
The objective function \eqref{eq:rMIRT_objective_function} can be written as a quadratic form towards either the reconstruction variable $\bm{x}$ in the convex domain $\left[0, 1 \right]^n$ when $\bm{\alpha}$ and $\bm{\bm{p}}$ are fixed:
    \begin{equation}
        \restr{f}{\bm{\alpha}, \bm{p}}(\bm{x}) = \frac{1}{2} \left \| \bm{W} \bm{P} \left(\bm{\alpha}, \bm{p}\right) \bm{x} - \bm{b} \right \|_2^2,
    \end{equation}
with 
    \begin{equation}
    \bm{P} \left(\bm{\alpha}, \bm{p}\right) = \left[\left( \bm{M}(\bm{p}) - \bm{I} \right) \diag \left\{ \bm{\alpha} \right \} + \bm{I} \right]
    \underbrace{
    \begin{bmatrix}
        I \\
        I \\
        \vdots \\
        I
    \end{bmatrix}}_\text{$n$ blocks $I$}.
    \end{equation}
Similarly, in the extended convex domain  $\left[0, 1 \right]^{nN}$ of $\bm{\alpha}$ when $\bm{x}$ and $\bm{p}$ are fixed, it yields:
    \begin{equation}
        \restr{f}{\bm{x}, \bm{p}}(\bm{\alpha}) = \frac{1}{2} \left \| \bm{W} \bm{Q} \left(\bm{x}, \bm{p} \right) \bm{\bm{\alpha}} - \bm{b} \right \|_2^2,
    \end{equation}
with 
    \begin{equation}
        \bm{Q} \left(\bm{x}, \bm{p} \right) = \left[\bm{M}(\bm{p}) - \bm{I} \right] \underbrace{
    \begin{bmatrix}
    \diag \left\{\bm{x} \right \} & 0 & 0 \\
    0 & \ddots & 0 \\
    0 & 0 & \diag \left\{\bm{x} \right \}
    \end{bmatrix}}_\text{$n$ blocks $\diag \left\{\bm{x} \right \}$}.
    \end{equation}
Consequently, it is biconvex towards $\bm{x}$ and $\bm{\alpha}$.
\end{proof}

\section{Experiment and results}
\begin{figure*}[!ht]
    \begin{subfigure}{.24\linewidth}
        \includegraphics[width=\linewidth]{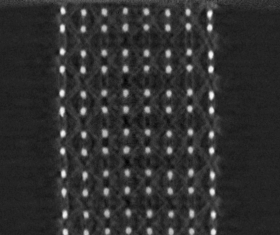}
        \caption{\centering Ground truth \newline}
    \end{subfigure}
    \begin{subfigure}{.24\linewidth}
        \includegraphics[width=\linewidth]{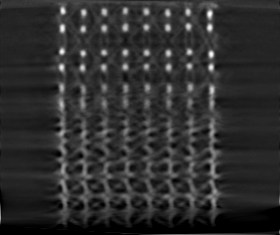}
        \caption{\centering without motion compensation \newline}
    \end{subfigure}
    \begin{subfigure}{.24\linewidth}
        \includegraphics[width=\linewidth]{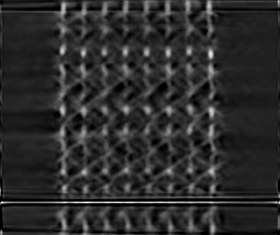}
        \caption{\centering without region-based motion compensation}
    \end{subfigure}
    \begin{subfigure}{.24\linewidth}
        \includegraphics[width=\linewidth]{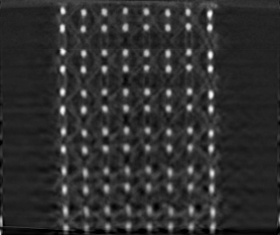}
        \caption{\centering with region-based motion compensation}
    \end{subfigure}
    \caption{\centering x-z cross-section of the reconstructions of the bone scaffold.}
\label{fig:recs}
\end{figure*}

We use a cylindrical bone scaffold of volume size $235 \times 280 \times 280$ (voxel) reconstructed from a real scan as the reference object. Projection data is simulated by generating $720$ uniformly-sampled cone beam projections spread over a full-rotation angular range. Gaussian noise with standard deviation of 1\% of the peak gray value of the projection data is added to the sinogram. We assume the object region from the top to the 50$^{th}$ horizontal cross-section slice to be static in all angular projections. The projections are captured at discrete angular time points and the motion is simulated as continuous constant scaling in all three dimensions y-z, x-z and x-y respectively with the scaling factors range from 1 to 0.99, 0.99 and 1.25 on the deformed area.

The experiment is considered in 5 subscans. We use the mean squared error to evaluate the reconstruction quality. The initial guess of the deformed region is the upper part of the object whose bottom z-coordinate is 55. At the $i^{th}$ iteration, the stepsizes $\gamma_{\bm{x}}^i$ and $\gamma_{\bm{p}}^i$ are chosen following the Barzilai-Borwein formula \cite{BB88} and the stepsize $\gamma_{\bm{\alpha}}^i$ is chosen constantly proportional to the quantity $1/i$. The z-coordinate of the center of motion is updated to be the z-coordinate of the bottom non-zero voxel of the intermediate estimated region encoder, when the x- and y- coordinates are in the center of the volume geometry. Convergence is achieved after around 15 iterations with a computation time of approximately 10 seconds per iteration. The reconstruction results are shown in \reffig{fig:recs}, and the behavior of the mean squared error is given in \reffig{fig:MSEs}. The reconstruction result of our method shows a clear improvement over the reconstruction without motion compensation and the reconstruction without region-based motion compensation \cite{Nguyen22_SPIE} wherein the deformation is supposed to appear in the entire volume of the object.

\begin{figure}[!ht]
    \includegraphics[width=\linewidth]{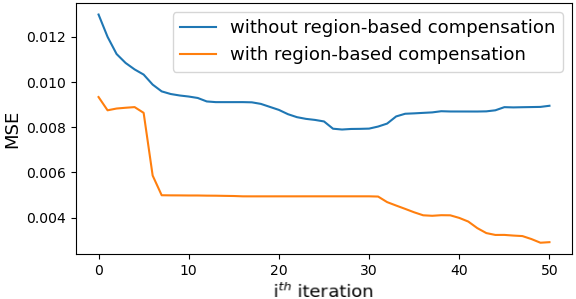}
    \caption{\centering Mean squared error of the reconstructions as a function
of the number of iterations.}
\label{fig:MSEs}
\end{figure}

\section{Conclusion and future work}
We have presented a reconstruction algorithm that combines accurate reconstruction, locally affine-deformed region identification and motion parameter estimation. Our method obtained a reconstruction result improved from the result of the reconstruction without motion compensation and the result of the reconstruction with the compensation of motions in the entire object volume \cite{Nguyen22_SPIE}. In future research, we aim to validate our method on real datasets and quantitatively compare the reconstruction and motion estimation result with the results from state-of-the-art methods. 


\section{Acknowledgement}
This study is partially supported by the Research Foundation-Flanders (FWO) (SBO grant no. S007219N and PhD grant no. 1SA2920N). The authors would like to thank Prof. Martine Wevers and Dr. Jeroen Soete for sharing the dataset.

\section{Compliance with ethical standards}
This is a numerical simulation study for which no ethical approval was required.

\bibliographystyle{IEEEbibetal}
\bibliography{references}

\end{document}